\documentclass[12pt]{amsart}

\usepackage{amsthm,amsfonts,amsbsy,amssymb,amsmath,amscd}
\usepackage[all]{xy}
\usepackage[colorlinks=true]{hyperref}

\begin{document}
\title{Severi inequality for varieties of maximal Albanese dimension}
\author{Tong Zhang}
\date{\today}

\address{Department of Mathematics, University of Alberta, Edmonton, Alberta T6G 2G1, Canada}
\email{tzhang5@ualberta.ca}
\maketitle

\theoremstyle{plain}
\newtheorem{theorem}{Theorem}[section]
\newtheorem{lemma}[theorem]{Lemma}
\newtheorem{coro}[theorem]{Corollary}
\newtheorem{prop}[theorem]{Proposition}
\newtheorem{defi}[theorem]{Definition}

\theoremstyle{remark}\newtheorem{remark}[theorem]{Remark}

\tableofcontents

\section{Introduction}
\subsection{Main result}
We work over the complex number field $\mathbb C$. Let $X$ be a projective and irregular variety over $\mathbb C$, i.e., $h^1(\mathcal O_X)>0$.
We say that $X$ is \textit{of maximal Albanese dimension}, if the image of $X$ under its Albanese map has the same dimension as $X$ itself.

A typical example of a variety of maximal Albanese dimension which is also of general type
can be constructed as follows:
fix an Abelian variety $A$ of dimension $n$ and
a very ample line bundle $H$ on $A$. Let $X$ be the double cover of $A$ branched along a smooth divisor $L$ such that $L \sim_{\rm{lin}} 2H$.
It is easy to see that $X$ is of maximal Albanese dimension and $X$ is of general type since $K_X$ is ample. Moreover, from the double cover formula, we have
$$
K^n_X=2H^n.
$$
By the Hirzebruch-Riemann-Roch formula,
$$
\chi(K_X)=(-1)^n\chi(\mathcal O_X)=\chi(H)=\frac{H^n}{n!}.
$$
Therefore, we have
$$
K^n_X = 2 n!\chi(K_X).
$$

In this paper, we prove the following theorem, which shows that $2 n!\chi(K_X)$ turns out to be the optimal lower bound of $K^n_X$.
\begin{theorem} \label{severi}
Let $X$ be a projective, normal, minimal and Gorenstein $n$-dimensional variety of general type. Suppose $X$ is of maximal Albanese dimension. Then
$$
K^n_X \ge  2 n! \chi(K_X).
$$
\end{theorem}

This theorem was known previously for small $n$.
For example, when $n=1$, the result just says that $\deg K_X=2g(X)-2$ for a smooth curve $X$ of positive genus. When $n=2$, the theorem, known as Severi inequality for surfaces, is proved by Pardini in \cite{Pa} using the slope inequality in \cite{CH, Xi1} and a clever covering (Albanese lifting) method. Pardini's theorem is reproved in \cite{YZ2} using the relative Noether inequality for fibered surfaces. Moreover, it is proved in \cite{YZ2} that the theorem also holds in positive characteristic.

We point out that Theorem \ref{severi} is also independently proved by Barja \cite{Ba} during the preparation of this paper. His proof
also relies on Pardini's method. The difference between our approaches is that, he uses a suitable version of Xiao's method \cite{Xi1} on
the Harder-Narasimhan filtration and continuous linear series developed by Mendes Lopes, Pardini and Pirola \cite{MPP}, while we use the relative Noether inequality which can be viewed as a generalization of \cite{YZ2}
in the surface case. 

Also, it should be pointed out that using the relative Noether inequality, we can get an upper bound of $h^0(L)$ for arbitrary nef $L$, not only the continuous rank of $L$ (see \cite{Ba}). In particular, using the relative Noether inequality for fibered surfaces, we can get the slope inequality of Cornalba-Harris-Xiao in arbitrary characteristic \cite{YZ2}, although it only involves some basic theories on linear series.
Furthermore, one should note that the relative Noether formula here still holds in positive characteristic. In an ongoing paper joint with Yuan, we also consider the arithmetic version of this inequality in Arakelov geometry.

\subsection{Idea of the proof}
As pointed before, the idea is similar to \cite{YZ2}. We first prove the relative Noether inequality for fibered $n$-folds by induction on the dimension
(see Proposition \ref{strongrelnoether}). As a result, we obtain a Hilbert-Samuel type result as Corollary \ref{essential}.
Finally, applying Corollary \ref{essential} to the $d$-th Albanese lifting of (a smooth model of) $X$,
we can get Theorem \ref{severi} by taking the limit. 

However, when we take the limit, the intermediate cohomologies are involved. Therefore,
we also need to control the asymptotic behavior of such cohomologies. Fortunately, it
can be implied by the generic vanishing theorem \cite{GL} (see Theorem \ref{albanese}), with which we can finish the whole proof.

{\bf Acknowledgement.}
The author would like to thank Xi Chen for a lot of helpful discussions.
Part of the idea in this paper comes from the joint papers of the author with Xinyi Yuan on Arakelov geometry and algebraic geometry \cite{YZ1, YZ2}.
The author is very grateful to Xinyi Yuan for his inspiring communications.

\section{Preliminaries}
We first assume $X$ to be a
smooth projective variety of dimension $n$.
Let $f: X \to Y$ be a fibration from $X$ to a smooth projective curve $Y$ with smooth general fiber
$F$.
The above notations will be used throughout Section 2 and 3.

\subsection{The reduction process}
For any nef line bundle $L$ on $X$, we
can find an integer $e_L > 0$ such that
\begin{itemize}
\item $L-e_LF$ is not nef;
\item $L-eF$ is nef for any integer $e<e_L$.
\end{itemize}

We have the following lemma.
\begin{lemma} \label{horizontal}
Let $L$ be a nef line bundle on $X$ such that $h^0(L-e_LF) > 0$.
Then the base locus of the linear system $|L-e_LF|$ has horizontal part.
\end{lemma}

\begin{proof}
By the definition of $e_L$, $L-e_LF$ is not nef. Then there
exists an irreducible and reduced curve $C$ on $X$ such that
$$
(L-e_LF)C<0.
$$
Hence $C$ is contained in the base locus of $|L-e_LF|$. It suffices to prove that $C$ is horizontal. Since $L$ is nef, $LC \ge 0$. It gives
$FC > 0$, which implies that $C$ can not be vertical.
\end{proof}

\begin{lemma} \label{decomposition1}
Let $L$ be a line bundle on $X$ such that $|L|$ is base point free and that $h^0(L-e_LF) > 0$.
Then we have the following decomposition:
$$
\pi^*(L-e_LF)= L_1 + Z_1.
$$
Here $\pi : X_1 \to X$ is a composition of blow-ups of $X$ such that the proper transformation of $|L-e_LF|$ is base point free.
$L_1$ and $Z_1$ are line bundles on $X_1$.
Moreover, we have
\begin{itemize}
\item[(1)] $|L_1|$ is base point free;
\item[(2)] $Z_1 \ge 0$;
\item[(3)] $h^0(L_1) < h^0(L)$.
\end{itemize}
\end{lemma}

\begin{proof}
Let $L_1$ be the proper transformation of $|L-e_LF|$ on $X_1$. It is automatically base point free. Now $Z_i$ is the fixed part of $|\pi^*(L-e_LF)|$. So it is effective. Furthermore, since $|L|$ is base point free and $e_L>0$,
it follows that
$$
h^0(L) > h^0(L-e_LF) = h^0(L_1).
$$
\end{proof}

We have the following general theorem.
\begin{theorem}\label{decomposition}
With the above notations. Let $L$ be a nef line bundle on $X$.
Then we have the following quadruples
$$
\{(X_i, L_i, Z_i, a_i), \quad i=0, 1, \cdots, N\}
$$
with the following properties:
\begin{enumerate}
\item $(X_0, L_0, Z_0, a_0)=(X, L, 0, e_L)$.
\item For any $i=0, \cdots, N-1$, $\pi_i: X_{i+1} \to X_i$ is a composition of blow-ups of $X_i$ such that the proper transformation of $|L_i-a_iF_i|$ is base point free. Here $F_0=F$, $F_{i+1}=\pi^*_{i} F_{i}$ and $a_i=e_{L_i}$. Moreover, we have the decomposition
$$
\pi^*_i(L_i-a_iF_i)= L_{i+1} + Z_{i+1}
$$
such that $|L_{i+1}|$ is base point free and $Z_{i+1} \ge 0$.
\item We have $h^0(L_0) > h^0(L_1) > \cdots > h^0(L_N)>h^0(L_N-a_NF_N)=0$. Here $a_N=e_{L_N}$.
\end{enumerate}
\end{theorem}

\begin{proof}
The quadruple $(X_{i+1}, L_{i+1}, Z_{i+1}, a_{i+1})$ is obtained by applying Theorem \ref{decomposition1} to
$(X_i, L_i, Z_i, a_i)$. The whole process terminates because $h^0(L_i)$ decreases strictly as $i$ goes larger and they are non-negative.
\end{proof}

\subsection{Numerical inequalities}
Resume the notations in Theorem \ref{decomposition}.
Denote
$$
L'_i=L_i-a_iF_i
$$
for $i=0, \cdots, N$.
Note that we have
$L_i|_{F_i}=L'_i|_{F_i}$ following from the construction.
Denote
$$
r_i=h^0(L_i|_{F_i}), \quad d_i=(L_i|_{F_i})^{n-1}.
$$

From our construction, we know that $L'_i+F_i$ is nef for each $i=0, \cdots, N$. Therefore, we have
\begin{lemma} \label{easylemma}
With the above notations, for any $i=0, \cdots, N$,
$$
L'^n_i+nd_i=(L'_i+F_i)^n \ge 0.
$$
\end{lemma}

\begin{prop} \label{algcase1}
For any $j=0, 1, \cdots, N$, we have the following numerical inequalities:
\begin{itemize}
\item[(1)] $\displaystyle{h^0(L_0) \le h^0(L'_j) + \sum_{i=0}^j a_{i} r_{i}}$;
\item[(2)] $\displaystyle{L^n_0 \ge n(a_0-1)d_0 + n \sum_{i=1}^j a_i d_i}$. \\
\end{itemize}
\end{prop}

\begin{proof}
We have the following exact sequence:
$$
0 \longrightarrow H^0(L_{i+1}-F_{i+1}) \longrightarrow H^0(L_{i+1}) \longrightarrow H^0(L_{i+1}|_{F_{i+1}}).
$$
Then it follows that
$$
h^0(L_{i+1}-F_{i+1}) \le h^0(L_{i+1})-h^0(L_{i+1}|_{F_{i+1}})=h^0(L_{i+1})-r_{i+1}.
$$
So by induction, we have
$$
h^0(L'_{i+1})=h^0(L_{i+1}-a_{i+1}F_{i+1}) \le h^0(L_{i+1})-a_{i+1}r_{i+1} = h^0(L'_{i})-a_{i+1}r_{i+1}.
$$
Furthermore,
$$
h^0(L_0) \le h^0(L'_0)+a_0r_0.
$$
Hence (1) is proved by summing over $i=0, \cdots, j-1$.

To prove (2), note that both $L'_i+F_i$ and $L'_{i+1}+F_{i+1}$ are nef, and $Z_i$ is effective. We get
\begin{eqnarray*}
L'^n_{i}-{L'^n_{i+1}} & = &(\pi_i^*L'_i-L'_{i+1}) \sum_{p=0}^{n-1}   (\pi^*_iL'_i)^{n-1-p}  {L'^p_{i+1}}\\
&=& (a_{i+1}F_{i+1}+Z_{i+1}) \sum_{p=0}^{n-1}   (\pi^*_iL'_i)^{n-1-p}  {L'^p_{i+1}} \\
&=& a_{i+1} \sum_{p=0}^{n-1}   (\pi^*_iL'_i)^{n-1-p}  {L'^p_{i+1}}F_{i+1} \\
&& + \sum_{p=0}^{n-1}   (\pi^*_iL'_i + F_{i+1})^{n-1-p}  (L'_{i+1}+F_{i+1})^p Z_{i+1}\\
&&-\sum_{p=0}^{n-1} ( p ( \pi^*L'_i)^{n-1-p} L'^{p-1}_{i+1} + (n-p-1)(\pi^*L'_i)^{n-2-p} L'^{p}_{i+1}) )  F_{i+1}Z_{i+1} \\
& \ge & n a_{i+1} d_{i+1} - n \sum_{p=0}^{n-2}   (\pi^*_iL'_i)^{n-2-p} {L'^p_{i+1}} F_{i+1}Z_{i+1} \\
& = & n a_{i+1} d_{i+1} - n (\pi_i^*L'_i-L'_{i+1}) F_{i+1}\sum_{p=0}^{n-2}   (\pi^*_iL'_i)^{n-2-p} {L'^p_{i+1}} \\
& = & n a_{i+1} d_{i+1} -n(d_i-d_{i+1})
\end{eqnarray*}

Summing over $i=0, 1, \cdots, j-1$, we have
$$
{L'^n_0}  \ge  {L'^n_{j}} + n\sum_{i=1}^j a_{i}d_i- n(d_0-d_j).
$$
Since
$$
L^n_0-{L'^n_0}=na_0d_0.
$$
Then (2) follows by applying Lemma \ref{easylemma}.
\end{proof}

We also have the following lemma.
\begin{lemma}\label{algsumai}
With the above notations, we have
$$
L^n_0 \ge d_0(na_0+\sum_{i=1}^N a_i -n).
$$
\end{lemma}

\begin{proof}
For $i=0, \cdots, N-1$, denote by
$$
\tau_i=\pi_{i} \circ \cdots \circ \pi_{N-1} : X_N \to X_i
$$
the composition of blow-ups and denote $\tau_N = {\rm id}_{X_N}: X_N \to X_N$. 

Write $b=a_1+\cdots+a_N$ and $Z=\tau^*_1Z_1+ \cdots +\tau^*_NZ_N$. We have the following numerical equivalence on $X_N$:
$$
\tau^*_0 L'_0 \sim_{\rm{num}} L'_N + bF_{N} + Z.
$$
Since $L'_0+F_0$ and $L'_N+F_N$ are both nef, it follows that
\begin{eqnarray*}
(L'_0+F_0)^n & = &(\tau^*_0 L'_0 + F_{N})^{n-1}(L'_N+F_N+bF_N+Z) \\
& \ge & (\tau^*_0 L'_0 + F_{N})^{n-1}(L'_N+F_N) + b(\tau^*_0 L'_0 + F_{N})^{n-1} F_N \\
& \ge & b d_0.
\end{eqnarray*}
Combining with
$$
L_0^n-(L'_0+F_0)^n=n(a_0-1)d_0,
$$
the proof is finished.
\end{proof}

We end up this section with a remark that when $X$ is a fibered surface, or even $X$ is an arithmetic surface,
all the above results have been studied in \cite{YZ1, YZ2}.

\section{Relative Noether inequality}
In this section, we will prove the relative Noether type inequality.

We say a divisor $D$ on a variety $X$ of dimension $n$ is \textit{pseudo-effective}, if for any nef line bundles $A_1, \cdots, A_{n-1}$ on $X$, we have
$$
A_1 \cdots A_{n-1} D \ge 0.
$$

\begin{prop} \label{strongrelnoether}
Let $f: X \to Y$ be a fibration from $X$ to a smooth curve $Y$ with the general smooth fiber $F$ of general type.
Suppose that $L$ is a nef line bundle on $X$.
Fix a line bundle $B$ on $F$ such that
\begin{itemize}
\item $\delta_i=(L|_F)^{n-i-1} B^i>0$ for any $i \ge 0$,
\item $|B|$ is base point free on $F$,
\item $k B-L|_F$ is pseudo-effective on $F$ for certain integer $k > 0$,
\item $L|_F \le K_F$.
\end{itemize}
Then one has
$$
h^0(L) - \frac {1}{2 n!} L^n \le c_n \left(\frac{L^n}{\delta_0}+1\right) \sum_{j=1}^{n-1}k^{j-1}\delta_j + \frac {\delta_0}{2(n-1)!}.
$$
Here $c_n \ge 1$ is a constant depending only on the number $n$.
\end{prop}
We point out that such $k$ always exists, because $B$ is big on $F$.

\begin{proof}
Our proof here is by induction. When $n=2$, since $\delta_0, \delta_1 > 0$,
the result just says $h^0(L)$ can be bounded in terms of $L^2$. In fact, if $\delta_0 \ge 2$,
from the relative Noether inequality in \cite{YZ2}, we have
$$
h^0(L) \le (\frac 14 + \frac {3}{4\delta_0}) L^2 + \frac {\delta_0+3}{2}.
$$
If $\delta_0=1$, then we can use the above inequality to bound $h^0(2L)$, which is enough to give a bound of $h^0(L)$.

Now we assume that the result holds for fibered varieties of dimension $\le n-1$ $(n \ge 3)$.

Resume the notations in Theorem \ref{decomposition}. By Proposition \ref{algcase1}, one has
\begin{eqnarray*}
h^0(L_0) & \le & h^0(L'_N) + \sum_{i=0}^N a_{i} r_{i}; \\
L^n_0 & \ge & n(a_0-1)d_0 + n\sum_{i=1}^N a_{i} d_i.
\end{eqnarray*}
Here $d_0=\delta_0$.
In the following, we will use induction to compare $r_i$ and $d_i$.

For any $i=1, \cdots, N$, write
$$
\rho_{i}=\pi_{0} \circ \cdots \circ \pi_{i-1}: X_i \to X_0
$$
and $\rho_0 = {\rm id}_{X_0}: X_0 \to X_0$.
Denote $B_0=B$ and $B_{i}=\rho^*_{i}B$.

\textbf{Case I}.
If $d_i>0$, choose two general members $B_{1, i}, B_{2, i} \in |B_i|$. Let $\sigma: \widetilde{F_i} \to F_i$ be the blow-up along their intersection. We get a fibration $f_i: \widetilde{F_i} \to \mathbb P^1$.
Denote the general fiber of $f_i$ by $\widetilde{B_i}$. Since $L_i|_{F_i}$ is big on $F_i$ and $L_i \le \rho^*_i L_0$, we can easily check the following facts:
\begin{itemize}
\item For each $j=1, \cdots, n-1$, we have
$$
\delta_j \ge (L_i|_{F_i})^{n-1-j} B^{j}_i=(\sigma^*L_i|_{\widetilde{B_i}})^{n-1-j} (\sigma^*B_i|_{\widetilde{B_i}})^{j-1}  > 0;
$$
\item $|\sigma^*B_i|_{\widetilde{B_i}}|$ is base point free;
\item $\sigma^*(kB_i-L_i)|_{\widetilde{B_i}}$ is pseudo-effective;
\item Since $L_i|_{F_i} \le \rho^*_i L|_{F_i} \le \rho^*K_F \le K_{F_i} $, by adjunction formula, one has
$$
\sigma^*L_i|_{\widetilde{B_i}} \le K_{\widetilde{B_i}}.
$$
\end{itemize}
By induction and using the fact that
$$
\frac{d_i}{(\sigma^*L_i|_{\widetilde{B_i}})^{n-2}} = \frac{(L_i|_{F_i})^{n-1}}{(L_i|_{F_i})^{n-2}B_i} \le \frac{k(L_i|_{F_i})^{n-2}B_i}{(L_i|_{F_i})^{n-2}B_i} =k,
$$
we have
$$
r_i \le \frac{1}{2(n-1)!} d_i + 2 c_{n-1} \sum_{j=1}^{n-1}k^{j-1}\delta_j.
$$

\textbf{Case II}.
If $d_i=0$, it implies that $L_i|_{F_i}$ is not big. Therefore
$$
r_i = h^0((L_i|_{F_i})|_{B_i}) \le h^0(L_0|_{B}).
$$
Note that $L_0|_{B}$ is nef and big on $B$,
$(L_0|_{B})^{n-2}=\delta_{1}$. So if $n-2 \ge 2$, we can use the blow-up trick as in Case I and get the following inequality by induction:
$$
r_i \le \frac{1}{2(n-2)!} \delta_{1} + 2 c_{n-2} \sum_{j=2}^{n-1}k^{j-2}\delta_j.
$$
The only problem is when $n=3$. However, in this case, the above inequality is nothing but Clifford's inequality.

As a result of the above two case, we are safe to use the following inequality:
$$
r_i \le \frac{1}{2(n-1)!} d_i + 2(c_{n-1}+c_{n-2}) \sum_{j=1}^{n-1}k^{j-1}\delta_j.
$$

Using the above comparison and Lemma \ref{algsumai}, it follows that
\begin{eqnarray*}
h^0(L_0) - \frac {1}{2 n!} L^n_0 & \le & 2(c_{n-1}+c_{n-2}) \sum_{j=1}^{n-1}k^{j-1}\delta_j\sum_{i=0}^N a_i + \frac {\delta_0}{2(n-1)!} \\
& \le & 2(c_{n-1}+c_{n-2}) \sum_{j=1}^{n-1}k^{j-1}\delta_j \left(\frac{L^n_0}{\delta_0}+1\right) + \frac {\delta_0}{2(n-1)!}
\end{eqnarray*}
Hence our result is proved by letting $c_n=2(c_{n-1}+c_{n-2})$.
\end{proof}

As a corollary, we have the following Hilbert-Samuel type result.
\begin{coro}\label{essential}
Let $X$ be an $n$-dimensional smooth projective variety of general type
and $L$ be a nef line bundle on $X$.
Fix a line bundle $B$ on $X$ such that
\begin{itemize}
\item $L \le K_X$;
\item $\delta_i=L^{n-i}B^i \ge 0$ for any $1 \le i \le n-1$;
\item $|B|$ is base point free on $X$ and $\delta_n=B^n>0$;
\item $k B-L$ is pseudo-effective on $X$ for certain integer $k > 0$.
\end{itemize}
Then one has
$$
h^0(L) - \frac {1}{2 n!} L^n \le C_n \sum_{j=1}^{n}k^{j-1}\delta_j.
$$
Here $C_n$ is a constant depending only on the number $n$.
\end{coro}
\begin{proof}
The idea is very similar to the proof of Proposition \ref{strongrelnoether}. In fact, in the previous proof,
we have shown how to prove this result in lower dimensional case. Here we only need to apply the blow-up trick once more.

Suppose $L^n>0$. Then $L$ is big. Hence $\delta_i>0$ for $i \ge 1$. Therefore, using the above blow-up trick on $X$, 
and the conclusion follows. 

If $L^n=0$, i.e., $L$ is not big. Since $B$ is big, we have
$h^0(L)=h^0(L|_B)$. By induction, we can find a positive integer $i_0 < n$ such that 
$$
h^0(L)=h^0(L|_{B^{i_0}}).
$$
Otherwise $h^0(L) \le 1$. But here we can apply the blow-up trick on $B^{i_0}$ to get the conclusion.
\end{proof}

At the end, we would like to mention that the same induction method can be applied to prove the relative Noether inequality in positive characteristic,
because the relative Noether inequality (resp. Clifford's inequality) still holds for Gorenstein fibered surfaces (resp. curves) \cite{Li2, YZ2}.

\section{Asymptotic behavior of cohomological dimensions}
Let $X$ be a projective and irregular variety. Let $A$ be the Albanese variety of $X$ of dimension $m=h^1(\mathcal O_X)>0$, and $a(X)$ be its Albanese image. Let $\mu_d: A \to A$ be the multiplicative map by $d$.
We have the following diagram:
$$
\xymatrix{X_d \ar[r]^{\phi_d} \ar[d]_{\alpha_d} & X \ar[d]^{\alpha={\rm Alb}_X} \\
A \ar[r]^{\mu_d} & A}
$$
Here $X_d=X \times_{\mu_d} A$ is the fiber product. We call $X_d$ \textit{the $d$-th Albanese lifting of $X$}.

In this section, we will prove

\begin{theorem} \label{albanese}
Let $X$ be a projective, smooth and irregular variety and $a(X)$ be its Albanese image. For each $d \in \mathbb N$, let $X_d$ be the $d$-th Albanese
lifting of $X$. Then for each $i=0, \cdots, \dim a(X)-1$, we have
$$
\lim_{d \to \infty} \frac{h^i(\mathcal O_{X_d})}{d^{2m}}=0
$$
Here $m=h^1(\mathcal O_X)$.
\end{theorem}

Before the proof, we recall the statement of the generic vanishing theorem first, which is
due to Green and Lazarsfeld \cite{GL}.

Let $X$ be a projective, smooth and irregular variety over $\mathbb C$. Let
$$
S^i(X)=\{ L \in {\rm Pic}^0(X) : h^i(L) > 0\}, \quad i=0, \cdots, \dim X,
$$
and $a(X)$ be the Albanese image of $X$. By the semi-continuity theorem,
we know that $S^i(X)$ is closed for each $i$.

We have the following remarkable theorem:
\begin{theorem} [Green-Lazarsfeld]
For any $i=0, \cdots, \dim X$,  we have
$$
{\rm codim} \{S^i(X), {\rm Pic}^0(X) \} \ge \dim a(X) - i.
$$
\end{theorem}

With the above theorem, we can prove Theorem \ref{albanese}.
\begin{proof} [Proof of Theorem \ref{albanese}]
Recall our construction of the $d$-th Albanese
lifting. The map $\phi_d: X_d \to X$ is a finite abelian cover. Moreover,
$$
{\rm Gal}(X_d/X)={\rm Gal}(\mu_d)=(\mathbb Z/d \mathbb Z)^{2m}.
$$
Thus we have
$$
(\phi_d)_* \mathcal O_{X_d} = \bigoplus_{L \in T_d} L,
$$
where $T_d \cong (\mathbb Z/d \mathbb Z)^{2m}$ is the subgroup of ${\rm Pic}^0(X)$ consisting of
all the $d$-torsion line bundles on $X$.

For each $i=0, \cdots, \dim a(X)-1$, it follows from the generic vanishing theorem that $h^i(L)=0$ for each $L \in {\rm Pic}^0(X) \setminus {S^i(X)}$.
Also from the semi-continuity theorem, we can find an $M_i>0$ such that
$$
M_i = \max_{L \in {\rm Pic}^0(X)} h^i(L).
$$
Then we have
$$
h^i(\mathcal O_{X_d}) = h^i((\phi_d)_* \mathcal O_{X_d}) = \sum_{L \in T_n} h^i(L) \le l_d M_i,
$$
where $l_d= \# (T_d \cap S^i(X))$. Because $S^i(X)$ has positive codimension, one has
$$
\lim_{d \to \infty} \frac{\# (T_d \cap S^i(X))}{\# T_d}=0.
$$
Hence
$$
\lim_{d \to \infty} \frac{h^i(\mathcal O_{X_d})}{d^{2m}} = 0.
$$
\end{proof}

\section{Proof of Theorem \ref{severi}}
In this section, we give the proof of Theorem \ref{severi} using Corollary \ref{essential} and Theorem \ref{albanese} by constructing the Albanese
lifting of $X$. This construction has been used by Pardini \cite{Pa} in the proof of the surface case.

Let $X$ be as in Theorem \ref{severi}. Let $\varepsilon: X' \to X$ be the resolution of singularities of $X$.
Since $X$ is minimal, $X$ has only terminal singularities \cite{KMM}. It is known that terminal singularities are rational, so
we have
$$
\chi(\mathcal O_X)=\chi(\mathcal O_{X'}), \quad h^0(K_X)=h^0(K_{X'}).
$$
Furthermore, since $X$ is of maximal Albanese dimension, so is $X'$.

Denote
${\rm Alb}_{X'}: X' \to A$ to be the Albanese map of $X'$, where $A={\rm Alb}(X')$ is the Albanese variety of $X'$ whose dimension is
$m=h^1(\mathcal O_X)$.

Let $H$ be a very ample line bundle on $A$, and $L$ be the pull-back of a very general member of $|H|$
on $X'$ by ${\rm Alb}_{X'}$. Set
$$
\delta_i=(\varepsilon^*K_{X})^{n-i}L^i
$$
for $i=0, \cdots, n$.
Since $L$ is big, we are able to find a positive integer $k$ such that $h^0(kL-K_{X'})>0$.

Let $X_d$ (resp. $X'_d$) be the $d$-th Albanese lifting of $X$ (resp. $X'$).
It follows that $X_d$ is still normal, minimal and Gorenstein. So  
we have
$$
\chi(\mathcal O_{X'_d})=\chi(\mathcal O_{X_d})=d^{2m}\chi(\mathcal O_{X}), \quad h^0(K_{X_d})=h^0(K_{X'_d}).
$$

Write $\varepsilon_d: X'_d \to X_d$ to be the resolution of singularities of $X_d$.
Recall the diagram
$$
\xymatrix{X'_d \ar[r]^{\phi_d} \ar[d]_{\alpha_d} & X' \ar[d]^{\alpha={\rm Alb}_{X'}} \\
A \ar[r]^{\mu_d} & A}
$$
We have the following numerically equivalence on $A$ \cite{BL}:
$$
\mu_d^*H \sim_{\rm num} d^2 H,
$$
which yields
$$
\phi^*_d L \sim_{\rm num} d^2 L_d.
$$
Here $L_d=\alpha^*_d H$. It follows that for any $i=0, \cdots, n$,
$$
(\varepsilon^*_dK_{X_d})^{n-i}L^i_d=(\phi^*_d(\varepsilon^*K_{X}))^{n-i}L^i_d=d^{2m-2i}\delta_i.
$$
Furthermore, from the above numerical equivalence, we know that
$$
d^2 k L_d - K_{X'_d} \sim_{\rm num} \phi^*_d(kL-K_{X'})
$$
is pseudo-effective. Since $X$ is minimal and $\varepsilon^*_d K_{X_d} \le K_{X'_d}$, it implies that $d^2 k L_d - \varepsilon^*_d K_{X_d}$ is pseudo-effective.

Now, apply Corollary \ref{essential} to $\varepsilon^*_d K_{X_d}$ and it follows that
\begin{eqnarray*}
h^0(K_{X'_d}) & \le & \frac {1}{2 n!} K^n_{X_d} + C_n \sum_{j=1}^{n}(d^2k)^{j-1}(\varepsilon^*_d K_{X_d})^{n-j}L^j_d \\
& \le & \frac {d^{2m}}{2 n!} K^n_{X} + C_n d^{2m-2}\sum_{j=1}^{n} k^{j-1} \delta_j.
\end{eqnarray*}
On the other hand,
\begin{eqnarray*}
h^0(K_{X'_d}) & = & \chi(K_{X'_d}) - \sum_{j=0}^{n-1} (-1)^{n-j} h^i(\mathcal O_{X'_d}) \\
&=&d^{2m}\chi(K_{X}) - \sum_{j=0}^{n-1} (-1)^{n-j} h^j(\mathcal O_{X'_d}).
\end{eqnarray*}
So the proof of Theorem \ref{severi} is completed by letting $d \to \infty$ and applying Theorem \ref{albanese}.

\end{document}